\documentclass[a4paper, 12pt]{article}
\usepackage[utf8]{inputenc}
 \usepackage[top=2.5cm,bottom=3.0cm,left=2.5cm,right=2.5cm]{geometry}
\usepackage{amssymb,amsmath,amsfonts,amsthm}
\usepackage{graphics}
\usepackage{graphicx}  

\usepackage{color}
\usepackage{soul}

\usepackage{verbatim}

\usepackage{lipsum}

\usepackage{cancel}

 \usepackage[pagewise]{lineno} 

\providecommand{\keywords}[1]
{
  \small	
  \textbf{\textit{Keywords---}} #1
}

\newtheorem{theorem}{Theorem}[section]
\newtheorem{lemma}[theorem]{Lemma}
\newtheorem{proposition}[theorem]{Proposition}
\newtheorem{corollary}[theorem]{Corollary}
\theoremstyle{definition}
\newtheorem{definition}[theorem]{Definition}

\newtheorem{remark}[theorem]{Remark}
 
\newtheorem*{theorem*}{Theorem} 
\newtheorem*{definition*}{Definition}
\newtheorem*{prop*}{Proposition}
\newtheorem*{corollary*}{Corollary}

\newtheorem{question}[theorem]{Question}

\newcommand{\C}{\mathbb{C}}

\newcommand{\N}{\mathbb{N}}

\newcommand{\Z}{\mathbb{Z}}

\newcommand{\SP}{\mathbb{S}}

\newcommand{\e}{\epsilon}

\newcommand{\bm}{\mathbf}

\newcommand{\sub}{\subseteq}

\newcommand{\lav}{\left|}
\newcommand{\rav}{\right|}

\newcommand{\ldav}{\left| \left|}
\newcommand{\rdav}{\right| \right|}

\newcommand{\oo}{\infty}
\newcommand{\lc}{\left(} 
\newcommand{\rc}{\right)}
\newcommand{\lb}{\left[}
\newcommand{\rb}{\right]}
\newcommand{\lfp}{\left\{}
\newcommand{\rfp}{\right\}}
\newcommand{\mc}{\mathcal}
\newcommand{\mf}{\mathfrak}

\usepackage{hyperref}
\hypersetup{
    colorlinks,
    citecolor=black,
    filecolor=black,
    linkcolor=black,
    urlcolor=black
}

\title{On Generalized Hyperbolicity, Stability, and Shadowing for Linear Operators}

\date{ }

\author{
Ali Messaoudi\thanks{\texttt{ali.messaoudi@unesp.br}},
José T. Neto\thanks{\texttt{jose.tofanin@unesp.br}},
Manuel Saavedra\thanks{\texttt{manuel.saavmath@gmail.com}},
Ioannis Tsokanos\thanks{\texttt{ioannis.tsokanos@unesp.br}}\\[1ex]
\small Sao Paulo State University \\
\small Department of Mathematics, Institute of Biosciences, Letters and Exact Sciences\\
\small Rua Cristóvão Colombo, 2265, Jardim Nazareth, São José do Rio Preto,
15054-000, SP, Brazil
}

\begin{document}

\maketitle

\begin{abstract}

This work studies the relations between shadowing, topological
stability, Lipschitz structural stability, and pseudo-hyperbolicity for linear dynamical systems on Banach spaces. 

The main result establishes that pseudo-hyperbolicity implies both topological stability and strong Lipschitz structural stability, whereas strong Lipschitz structural stability implies the shadowing property. In addition, a spectral characterization of pseudo-hyperbolicity is obtained, yielding an equivalence between pseudo-hyperbolicity, topological stability, strong Lipschitz structural stability, and the shadowing property for invertible operators whose eigenspaces associated with unimodular eigenvalues admit closed complements. In particular, these four notions are equivalent on Hilbert spaces. 
Combined with a recent result of Dragičević and Pituk, these results yield, for a large class of Banach spaces, an invertible operator that has the shadowing property but is not generalized hyperbolic, thereby answering an open problem in linear dynamics. 

\end{abstract}


\smallskip 
\noindent 


\section{Introduction}

Hyperbolicity, shadowing, expansivity, and stability are fundamental notions in the theory of discrete dynamical systems; see, for instance \cite{Aoki-Topological_Dynamics, Katok-Hasselblatt-Introduction_Modern_Dynamical_Systems, Pilyugin-Shadowing_in_Dynamical_Systems, Shub-Global_Stability_of_Dynamical_Systems,  Walters_On_the_Pseudo_Orbit_Tracing_Property_and_its_relationship_to_Stability}.

In the setting of linear dynamics, the interplay among these notions has been extensively studied. Hartman \cite{Hartman} proved that every invertible hyperbolic operator on a finite-dimensional Banach space is strongly Lipschitz structurally stable. This result was independently extended to infinite-dimensional Banach spaces by Palis \cite{Palis} and Pugh \cite{Pugh}; see also \cite{Moser-On_A_Theorem_Of_Anosov}. The converse was established in finite dimensions by Robbin \cite{Robbin-Topological_Conjugacy_and_Structural_Stability_for_Discrete_Dynamical_Systems} in 1972. In contrast, Bernardes and Messaoudi \cite[Theorem 9]{Bernardes-Shadowing_and_Structural_Stability} recently showed that the converse fails in infinite-dimensional Banach spaces.

Motivated by these developments, attention has shifted to classes of operators beyond the hyperbolic setting. In particular, the notion now known as generalized hyperbolicity was introduced in \cite{Bernardes-Expansivity_and_Shadowing_in_Linear_Dynamics}, while the present terminology was adopted in \cite{Cirilo_Gollobit_Pujals-Dynamics_of_Generalized_Hyperbolic_Linear_Operators}. 
It is now known that generalized hyperbolicity implies strong Lipschitz structural stability \cite[Theorem~1]{Bernardes_Messaoudi-A_Generalized_Grobman-Hartman Theorem}, topological stability \cite{Lee_Morales-Topological_Stability_L-Shadowing_and_Generalized_Hyperbolicity}, and the shadowing property \cite{Bernardes-Expansivity_and_Shadowing_in_Linear_Dynamics}. Moreover, Lee and Morales proved that every topologically stable operator has the shadowing property \cite{Lee_Morales-Topological_Stability_L-Shadowing_and_Generalized_Hyperbolicity}; see also \cite[Corollary~15]{Bernardes_Caraballo_Darji_Favaro_Peris-Generalized_Hyperbolicity_Stability_and_Expansivity_for_Operators_on_Locally_Convex_Spaces}.

It is well known that an invertible operator $T$ on a complex Banach space is hyperbolic if and only if it is expansive and has the shadowing property \cite{Bernardes-Shadowing_and_Structural_Stability}. Moreover, in finite-dimensional spaces and for normal operators on Hilbert spaces, hyperbolicity, uniform expansivity, and the shadowing property are equivalent (see \cite{Bernardes-Expansivity_and_Shadowing_in_Linear_Dynamics,Eisenberg_Hedlund-Expansive_Automorphisms_of_Banach_Spaces,Ombach-The_Shadowing_Lemma_in_the_Linear_Case}). Corresponding spectral characterizations of hyperbolicity, uniform expansivity, and the shadowing property have also been established in terms of the spectrum \cite{Eisenberg_Hedlund-Expansive_Automorphisms_of_Banach_Spaces}, approximate point spectrum \cite{Hedlund-Expansive_Automorphisms_of_Banach_Spaces_II}, and surjective spectrum \cite{Dragicevic_Pituk-Duality_between_Shadowing_and_Uniform_Expansivity_in_Linear_Dynamics}, respectively.

For particular classes of operators on classical Banach spaces, related equivalences have already been established. Bayart \cite{Bayart-Two_Problems_on_Weighted_Shifts_in_Linear_Dynamics} proved that an invertible bilateral weighted shift with positive weights on $\ell^{p}(\mathbb{Z})$, $1\leq p<\infty$, or on $c_{0}(\mathbb{Z})$ is strongly Lipschitz structurally stable if and only if it has the shadowing property. Maiuriello \cite{Maiuriello-expan} studied strong Lipschitz structural stability for composition operators on $L_{p}$-spaces, $1\leq p<\infty$, showing that, in the dissipative case, the shadowing property implies strong Lipschitz structural stability and establishing sufficient conditions under which the two notions are equivalent. Moreover, D'Aniello, Darji, and Maiuriello
\cite[Corollary~GH]{DAniello_Darji_Maiuriello-Generalized_Hyperbolicity_and_Shadowing_in_Lp_Spaces}
proved that, for composition operators induced by dissipative systems of bounded distortion, generalized hyperbolicity is equivalent to the shadowing property.

In recent years, the shadowing property has also been investigated in connection with chaotic behavior in linear dynamics. For instance, Antunes et al. \cite{Antunes} proved that an operator on a Banach space that simultaneously possesses the shadowing property and chain recurrence is frequently hypercyclic and mixing. Bernardes and Peris \cite{Bernardes_Peris-On_Shadowing_and_Chain_Recurrence_in_Linear_Dynamics} subsequently showed that every operator in this class is densely distributionally chaotic.

\bigskip

Despite these advances, several fundamental questions concerning the precise relations among shadowing, stability, and chaotic phenomena remain unresolved. Two such problems will be addressed below:

\begin{question}\label{Q1}
Does strong Lipschitz structural stability imply the shadowing property? Conversely, does the shadowing property imply strong structural stability?
\end{question}

\begin{question}[\cite{Bernardes_Caraballo_Darji_Favaro_Peris-Generalized_Hyperbolicity_Stability_and_Expansivity_for_Operators_on_Locally_Convex_Spaces}, Problem~F]\label{Q2}
Is every operator with the shadowing property necessarily generalized hyperbolic?
\end{question}

\bigskip

In this work, the notion of pseudo-hyperbolicity, which is weaker than generalized hyperbolicity, is considered. Namely, an invertible operator $T$ on a complex Banach space $\mc{B}$ is said to be \emph{pseudo-hyperbolic} if there exist a bounded linear operator $P:\mc B\to\mc B$ and constants $C>0$ and $\beta\in(0,1)$ such that
\begin{equation}\label{Eq_Pseudo-Hyperbolicic Exponential Bounds}
\Vert T^n Px\Vert\le C\beta^n\Vert x\Vert
\quad\text{and}\quad
\Vert T^{-n}(I-P)x\Vert\le C\beta^n\Vert x\Vert,
\qquad \forall x\in\mc B\text{ and }n\ge0.
\end{equation}
The class of operators satisfying the exponential bounds \eqref{Eq_Pseudo-Hyperbolicic Exponential Bounds} has already appeared in the literature; see, for instance, the work of Pituk \cite{Pituk_Spectral Characterization of Shadowing for Linear Operators on Hilbert Spaces}.

Pseudo-hyperbolicity is shown to imply both strong Lipschitz structural stability and topological stability, while strong Lipschitz structural stability implies the shadowing property; see Theorem~\ref{Theor_Pseudo-Hyperbolicity Lipschitz Stability Shadowing}. Since every generalized hyperbolic operator is pseudo-hyperbolic, the same conclusions hold for generalized hyperbolicity. In particular, the implication from strong Lipschitz structural stability to the shadowing property gives an affirmative answer to the first part of Question~\ref{Q1}. 
A spectral characterization of pseudo-hyperbolicity is established in Proposition~\ref{Prop_Pseudo-Hyperbolicity Spectral Characterization}. Combining this characterization and Theorem~\ref{Theor_Pseudo-Hyperbolicity Lipschitz Stability Shadowing} with the recent spectral characterization of the shadowing property due to Dragičević and Pituk \cite{Dragicevic_Pituk-Duality_between_Shadowing_and_Uniform_Expansivity_in_Linear_Dynamics} yields the equivalence of pseudo-hyperbolicity, strong Lipschitz structural stability, topological stability, and the shadowing property for a broad class of operators, including all invertible operators on Hilbert spaces. 
Finally, for every $1<p<\infty$ with $p\neq2$, an invertible operator on $\ell^p(\mathbb{N})$ is constructed that has the shadowing property but is not generalized hyperbolic, thereby providing a negative answer to Question~\ref{Q2}. 
More generally, the construction applies to every Banach space that is isomorphic to the direct sum of itself with itself and admits a closed non-complemented subspace $M$ such that the quotient space is isomorphic to the original space. By the classical theorem of Lindenstrauss and Tzafriri \cite{Lindenstrauss_Tzafriri-On_the_Complemented_Subspaces_Problem}, the existence of a closed non-complemented subspace is equivalent to the space not being isomorphic to a Hilbert space.

\paragraph{}

The paper is organized as follows. Section
\ref{Sec_Preliminaries}
contains the necessary notions and definitions. The main results are stated in Section
\ref{Sec_Main Results}.
Their proofs are presented in Sections
\ref{Sec_Stability},
\ref{Sec_Pseudo-Hyperbolicity and Spectrum},
and
\ref{Sec_Example Shadowing but not Generalized Hyperbolic}.

\section{Preliminaries}\label{Sec_Preliminaries}

Throughout the paper, $\lc \mc B,\ldav \cdot \rdav_\mc{B}\rc$ denotes a complex Banach space. The Banach space of bounded linear operators on $\mathcal B$, endowed with the operator norm, is denoted by $\mathcal L(\mathcal B)$, while $\mathcal{GL}(\mathcal B)$ denotes its group of invertible elements. 

\paragraph{}
An operator $T\in\mc{GL}\lc\mc B\rc$ is said to be \emph{hyperbolic} if $\mc B$ admits a decomposition as the direct sum of two closed subspaces,
$\mc B=M\oplus N$, such that
\[
T(M)=M \quad\text{and}\quad T^{-1}(N)=N,
\]
and $\sigma\lc T|_M\rc\sub\mathbb D$ and $ \sigma\lc T^{-1}|_N\rc\sub\mathbb D$, where $\mathbb D=\lfp\lambda\in\C:\;|\lambda|<1\rfp $ denotes the open unit disk. The operator $T\in\mc{GL}(\mc B)$ is said to be \emph{expansive} if there exists a constant $C > 1$ such that, for every unit vector
$x\in\mc B$, there exists $n_x\in\Z$ such that $ \|T^{n_x}x\|_{\mc B} \ge C $. 
The operator $T$ is said to be \emph{uniformly expansive} if there exists an integer $n\ge1$ and a constant $C > 1$ such that, for every unit vector $x\in\mc B$, either
$$ 
\|T^n x\|_{\mc B} \ge C
\quad\text{or}\quad
\|T^{-n}x\|_{\mc B} \ge C.
$$
Infinite-dimensional examples of uniformly expansive operators that are not hyperbolic were first constructed by Eisenberg and Hedlund \cite{Eisenberg_Hedlund-Expansive_Automorphisms_of_Banach_Spaces}.

\paragraph{} 
The shadowing property (also known as the \emph{pseudo-orbit tracing property}) is a fundamental notion in dynamical systems; see, for instance, \cite{Pilyugin-Shadowing_in_Dynamical_Systems}. It was introduced independently by Bowen \cite{Bowen-Omega-limit_Sets_for_Axiom-A_Diffeomorphisms} and Sinai \cite{Sinai-Gibbs_Measures_in_Ergodic_Theory}.

\begin{definition}[Shadowing and Positive Shadowing Property]
Let $T\in\mc{GL}\lc\mc B\rc$ be an invertible operator. Given $\delta>0$, a sequence $(x_n)_{n\in\Z}$ is called a \emph{$\delta$-pseudo-orbit} of $T$ if
\[
\|x_{n+1}-Tx_n\|_{\mc B}\le\delta
\qquad\text{for every }n\in\Z.
\]
The operator $T$ is said to have the \emph{shadowing property} if, for every $\varepsilon>0$, there exists $\delta>0$ such that every $\delta$-pseudo-orbit $(x_n)_{n\in\Z}$ is \emph{$\varepsilon$-shadowed} by an orbit of $T$; that is, there exists $x\in\mc B$ such that $ \ldav x_n-T^n x\rdav_{\mc B}\le\varepsilon$ for every $n\in\Z$. 

For a non-invertible operator $T\in\mc L(\mc B)$, the notions of a
\emph{positive $\delta$-pseudo-orbit} and the \emph{positive shadowing property}
are defined analogously by replacing the index set $\Z$ with
$\N\cup\{0\}$. 
\end{definition}

Hyperbolicity always implies the shadowing property \cite[Theorem~1]{Ombach-The_Shadowing_Lemma_in_the_Linear_Case}. Motivated by the search for non-hyperbolic operators with the shadowing property, Bernardes \emph{et al.} \cite{Bernardes-Expansivity_and_Shadowing_in_Linear_Dynamics} introduced the notion of generalized hyperbolicity, thereby providing a broader class of operators with the shadowing property.

\begin{definition}[Generalized Hyperbolic Operators]\label{def-pseudo}
An operator $T\in\mc{GL}(\mc B)$ is said to be \emph{generalized hyperbolic} if $\mc B$ admits a decomposition $ \mc B=M\oplus N$, where $M$ and $N$ are closed subspaces satisfying $ T\lc M\rc \sub M$  and $ T^{-1}(N)\subset N$, 
and the spectra of the restrictions $T|_M$ and $T^{-1}|_N$ satisfy $ \sigma(T|_M) \sub \mathbb D $ and  $ \sigma(T^{-1}|_N) \sub \mathbb D$. 
\end{definition}

\noindent

Every generalized hyperbolic operator
$T\in\mc{GL}(\mc B)$
is pseudo-hyperbolic. Indeed, condition
\eqref{Eq_Pseudo-Hyperbolicic Exponential Bounds}
holds with $P$ being the projection of $\mc B$ onto $M$. Moreover, for this choice of $P$, the spectral inclusions $ \sigma(T|_M) \sub \mathbb D $ and  $ \sigma(T^{-1}|_N) \sub \mathbb D$ are equivalent to
\eqref{Eq_Pseudo-Hyperbolicic Exponential Bounds}.
In \cite[Theorem~A]{Bernardes-Expansivity_and_Shadowing_in_Linear_Dynamics},
Bernardes \emph{et al.} proved that every generalized hyperbolic operator has the shadowing property by exploiting condition
\eqref{Eq_Pseudo-Hyperbolicic Exponential Bounds}. 
Further examples of non-hyperbolic shadowing operators arising in this setting can be found in \cite{Aniello_Darji_Maiuriell-Generalized_Hyperbolicity_and_Shadowing_in_Lp_Spaces}, \cite{Cirilo_Gollobit_Pujals-Dynamics_of_Generalized_Hyperbolic_Linear_Operators}, and \cite[Example~1.9]{Pituk_Spectral Characterization of Shadowing for Linear Operators on Hilbert Spaces}.

\paragraph{}

Hyperbolicity, uniform expansivity, and the shadowing property admit spectral characterizations that play a central role in linear dynamics. 
To state the main results and review the relevant literature, let $ \SP^{1} = \lfp \lambda\in\C: |\lambda|=1 \rfp $ 
denote the unit circle, and define the \emph{spectrum}, \emph{point spectrum}, \emph{surjective spectrum}, \emph{approximate point spectrum}, and \emph{right spectrum} of an operator $T\in\mc L(\mc B)$ by 
\begin{align*}
\sigma(T)
&=
\lfp
\lambda\in\C:\;
T-\lambda I
\text{ is not invertible}
\rfp,\\
\sigma_{pt}(T)
&=
\lfp
\lambda\in\C:\;
T-\lambda I
\text{ is not injective}
\rfp,\\
\sigma_{sur}(T)
&=
\lfp
\lambda\in\C:\;
T-\lambda I
\text{ is not surjective}
\rfp,\\
\sigma_{ap}(T)
&=
\lfp
\lambda\in\C:\;
T-\lambda I
\text{ is not bounded below}
\rfp,\\
\sigma_r(T)
&=
\lfp
\lambda\in\C:\;
T-\lambda I
\text{ is not right invertible}
\rfp. 
\end{align*}
It is known that an operator $T$ is hyperbolic if and only if $ \sigma(T)\cap\SP^1=\emptyset$ (see \cite[Lemma~1]{Eisenberg_Hedlund-Expansive_Automorphisms_of_Banach_Spaces}). 
Likewise, \cite[Theorem 1]{Hedlund-Expansive_Automorphisms_of_Banach_Spaces_II} showed that $T$ is uniformly expansive if and only if $ \sigma_{ap}(T)\cap\SP^1=\emptyset $. 
The following theorem, recently obtained by Pituk and Dragičević \cite{Dragicevic_Pituk-Duality_between_Shadowing_and_Uniform_Expansivity_in_Linear_Dynamics}, characterizes the shadowing property in terms of the surjective spectrum (see also \cite{Pituk_Spectral Characterization of Shadowing for Linear Operators on Hilbert Spaces} for the Hilbert space case).

\begin{theorem}\label{Theor_Equivalence Shadowing and Spectrum}\cite[Theorem 2.2]{Dragicevic_Pituk-Duality_between_Shadowing_and_Uniform_Expansivity_in_Linear_Dynamics} 
Let $\mc{B}$ be a complex Banach space. An operator
$T\in G\mc{L}(\mc{B})$
(respectively, $T\in\mc{L}(\mc{B})$)
has the shadowing property
(respectively, the positive shadowing property)
if and only if $ \sigma_{sur}(T)\cap\SP^1=\emptyset $. 
\end{theorem}

\noindent
In Section \ref{Sec_Stability}, a shorter and independent proof than the one given in \cite{Dragicevic_Pituk-Duality_between_Shadowing_and_Uniform_Expansivity_in_Linear_Dynamics} is presented.

\paragraph{}

Spectral characterizations are only one aspect of the interplay between these dynamical properties. Equally important are their connections with structural and topological stability. 
The concept of stability was introduced by Andronov and Pontrjagin
\cite{Andronov_Pontrjagin-Systemes_Grossiers} 
and formalizes the idea that the qualitative behavior of a dynamical system should remain unchanged under sufficiently small perturbations. 
The invariance of qualitative behavior is formalized by the notion of topological conjugacy. More precisely, two continuous maps $f,g:\mc B\to\mc B$ are said to be \emph{topologically conjugate} if there exists a homeomorphism $H:\mc B\to\mc B$ such that $ f\circ H=H\circ g $.

The notions of topological stability and strongly Lipschitz structural stability are recalled below. To state the latter, given $\delta>0$, denote by
\[
\mf{Lip}_{\delta}\lc\mc B\rc=
\lfp
\phi:\mc B\to\mc B:
\ldav\phi\rdav_{\oo}\le\delta
\ \text{and}\
\operatorname{Lip}\lc\phi\rc\le\delta
\rfp
\]
the set of $\delta$-Lipschitz maps on $\mc B$, where
\[
\operatorname{Lip}\lc\phi\rc:=
\sup_{\substack{x,y\in\mc B\\ x\neq y}}
\frac{\ldav\phi(x)-\phi(y)\rdav_{\mc B}}
{\ldav x-y\rdav_{\mc B}}
\]
denotes the Lipschitz constant of $\phi$.

\begin{definition}[Topological Stability and strong Lipschitz Structural Stability]
An operator $T \in \mc{GL}\lc \mc{B} \rc$ is said to be \emph{topologically stable} if, for every $\e>0$, there exists $\delta>0$ such that, for every homeomorphism $g:\mc B\to\mc B$ satisfying $ 
\ldav T-g\rdav_{\oo}\le\delta$, there exists a continuous map $H:\mc B\to\mc B$ such that
\[
T\circ H=H\circ g
\qquad\text{and}\qquad
\ldav H-I\rdav_{\oo}\le\e.
\]

An operator $T\in\mc{GL}\lc\mc B\rc$ is said to be \emph{strongly Lipschitz structurally stable} if, for every $\e>0$, there exists $\delta>0$ such that, for every $\phi\in\mf{Lip}_{\delta}\lc\mc B\rc$, the operators $T$ and $T+\phi$ are topologically conjugate by a homeomorphism $H:\mc B\to\mc B$ satisfying $ \ldav H-I\rdav_{\oo}\le\e$.

\end{definition}

\section{Main results}\label{Sec_Main Results}

The first main result establishes the implications between pseudo-hyperbolicity, topological stability, strong Lipschitz structural stability, and the shadowing property. 

\begin{theorem}\label{Theor_Pseudo-Hyperbolicity Lipschitz Stability Shadowing}
Let $\mc{B}$ be a Banach space, and let $T\in\mc{GL}\lc\mc{B}\rc$. The following assertions hold:
\begin{enumerate}
    \item If $T$ is pseudo-hyperbolic, then $T$ is topologically stable and strongly Lipschitz structurally stable.

    \item If $T$ is strongly Lipschitz structurally stable, then $T$ has the shadowing property.
\end{enumerate} 
\end{theorem}

The following corollary gives an affirmative answer to
\cite[Question~15]{Bernardes-Shadowing_and_Structural_Stability}, which asks whether every expansive and strongly Lipschitz structurally stable operator is hyperbolic.

\begin{corollary}
Let $T\in\mc{GL}(\mc B)$. 
If $T$ is expansive and strongly Lipschitz structurally stable, then $T$ is hyperbolic.
\end{corollary}

\begin{remark}
If $T\in\mc{GL}(\mc B)$ is expansive, then hyperbolicity, pseudo-hyperbolicity, strong Lipschitz structural stability, topological stability, and the shadowing property are equivalent. 
Indeed, if $T$ is pseudo-hyperbolic, strongly Lipschitz structurally stable, topologically stable, or has the shadowing property, then by Theorem
\ref{Theor_Pseudo-Hyperbolicity Lipschitz Stability Shadowing} and \cite[Corollary~15]{Bernardes_Caraballo_Darji_Favaro_Peris-Generalized_Hyperbolicity_Stability_and_Expansivity_for_Operators_on_Locally_Convex_Spaces}, $T$ has the shadowing property. Hence, by Theorem
\ref{Theor_Equivalence Shadowing and Spectrum}, $ \sigma_{sur}(T)\cap\SP^{1}=\emptyset$. If $T$ were not hyperbolic, then $ \sigma(T)\cap\SP^{1}\neq\emptyset$. Since $ \sigma(T)=\sigma_{pt}(T)\cup\sigma_{sur}(T)$, it would follow that $ \sigma_{pt}(T)\cap\SP^{1}\neq\emptyset$, 
contradicting the fact that an expansive operator has no eigenvalues on the unit circle. 
\end{remark}

Moreover, Theorem
\ref{Theor_Pseudo-Hyperbolicity Lipschitz Stability Shadowing}
also recovers the following result.

\begin{corollary}\cite[Theorem~1]{Bernardes_Messaoudi-A_Generalized_Grobman-Hartman Theorem}
Let $T\in\mc{GL}(\mc B)$. If $T$ is generalized hyperbolic, then $T$ is strongly Lipschitz structurally stable.
\end{corollary}

\paragraph{}

The following theorem shows that, for a broad class of operators, pseudo-hyperbolicity, topological stability, strong Lipschitz structural stability, and the shadowing property are equivalent. To state the theorem, a subspace $\mc V\sub\mc B$ is said to admit a \emph{closed complement} if there exists a closed subspace $\mc M\sub\mc B$ such that $\mc B=\mc V\oplus\mc M$.

\begin{theorem}\label{Theor_Banach Equivalence Pseudo-Hyperbolicity Lipschitz Stability Shadowing} 
Let $\mc{B}$ be a Banach space. Assume that the invertible operator $T\in\mc{GL}\lc\mc B \rc$ is such that 
\begin{equation}\label{Eq_Assumption Closed Complement}
\ker\lc T-\lambda I\rc
\text{ has a closed complement for every }
\lambda\in\SP^1.
\end{equation}
Then the following statements are equivalent:
\begin{enumerate}
    \item $T$ is pseudo-hyperbolic.
    \item $T$ is topologically stable.
    \item $T$ is strongly Lipschitz structurally stable.
    \item $T$ has the shadowing property.
\end{enumerate} 

\noindent
In particular, if $\mc{B}$ is a Hilbert space, then condition \eqref{Eq_Assumption Closed Complement} holds.
\end{theorem}

Bilateral weighted shift operators form a class of operators on
$\ell^{p}\lc \Z \rc$, $1\le p<+\oo$, and $c_{0}\lc \Z \rc$ that satisfy condition
\eqref{Eq_Assumption Closed Complement}.
Here, the notation
$$
\ell^{p}\lc \Z \rc
=
\lfp
\lc x_{n}\rc_{n\in\Z}\in\C^{\Z}
:
\sum_{n\in\Z}\lav x_{n}\rav^{p}<+\oo
\rfp
\quad \text{and} \quad 
c_{0}\lc \Z \rc
=
\lfp
\lc x_{n}\rc_{n\in\Z}\in\C^{\Z}
:
\lav x_{n}\rav
\underset{|n|\to+\oo}{\longrightarrow}0
\rfp
$$
is used.
As an immediate consequence of Theorem
\ref{Theor_Banach Equivalence Pseudo-Hyperbolicity Lipschitz Stability Shadowing},
one recovers the following result due to Bayart
\cite{Bayart-Two_Problems_on_Weighted_Shifts_in_Linear_Dynamics}.

\begin{corollary}
Let $X=\ell^{p}(\Z)$, $1\le p<\infty$, or $X=c_{0}(\Z)$, and let
$(w_n)_{n\in\Z}\sub \C$ be a bounded sequence satisfying $ \inf_{n\in\Z}|w_n|>0$. 
Consider the bilateral weighted shift $T$ defined by
\[
T\bigl((x_n)_{n\in\Z}\bigr) 
=
(w_nx_{n+1})_{n\in\Z}.
\]
Then $T$ is strongly Lipschitz structurally stable if and only if it is shadowing.
\end{corollary}

\paragraph{} 

Theorem
\ref{Theor_Banach Equivalence Pseudo-Hyperbolicity Lipschitz Stability Shadowing}
will follow by combining
Theorem \ref{Theor_Pseudo-Hyperbolicity Lipschitz Stability Shadowing}
with the following characterization of pseudo-hyperbolicity in terms of the right spectrum. 

\begin{proposition}\label{Prop_Pseudo-Hyperbolicity Spectral Characterization}
Let $\mc{B}$ be a complex Banach space. An operator $T\in\mc{GL}\lc\mc{B}\rc$ is pseudo-hyperbolic if and only if $ \sigma_r(T)\cap\SP^1=\emptyset $. 
\end{proposition} 

\begin{remark}
A pseudo-hyperbolic operator $T$ is either hyperbolic or satisfies
$\SP^{1}\sub \sigma_{pt}(T)\setminus\sigma_{r}(T)$.
Indeed, the inclusion
$\partial\sigma(T) \sub \sigma_{r}(T)$,
combined with Proposition 
\ref{Prop_Pseudo-Hyperbolicity Spectral Characterization},
shows that
$\partial\sigma(T)\cap\SP^{1}=\varnothing$.
Since $\SP^{1}$ is connected, it follows that either
$\SP^{1}\sub \rho(T)$
or
$\SP^{1} \sub \operatorname{int}\sigma(T)$.
The former is precisely the condition of hyperbolicity. In the latter case, since
$\sigma(T)=\sigma_{pt}(T)\cup\sigma_{r}(T)$
and
$\sigma_{r}(T)\cap\SP^{1}=\varnothing$,
one obtains
$\SP^{1} \sub \sigma_{pt}(T)\setminus\sigma_{r}(T)$.
\end{remark}

Since $\sigma_{sur}\lc T\rc \sub \sigma_r\lc T\rc$, Proposition \ref{Prop_Pseudo-Hyperbolicity Spectral Characterization} and Theorem \ref{Theor_Equivalence Shadowing and Spectrum} recover the fact that every pseudo-hyperbolic operator (and, in particular, every generalized hyperbolic operator) has the shadowing property \cite[Theorem A]{Bernardes-Expansivity_and_Shadowing_in_Linear_Dynamics}.

\paragraph{}

Finally, combining Theorem \ref{Theor_Equivalence Shadowing and Spectrum} and Proposition \ref{Prop_Pseudo-Hyperbolicity Spectral Characterization} yields the following result, which resolves the open problem of whether every shadowing invertible operator is generalized hyperbolic \cite[Problem F]{Bernardes_Caraballo_Darji_Favaro_Peris-Generalized_Hyperbolicity_Stability_and_Expansivity_for_Operators_on_Locally_Convex_Spaces}.

\begin{theorem}\label{Theor_Shadowing but not Generalized Hyperbolic}
For every $1<p<+\oo$, $p\neq2$, there exists an invertible operator $T$ on $\ell^p(\mathbb N)$ which has the shadowing property but is not pseudo-hyperbolic. In particular, $T$ is not generalized hyperbolic.
\end{theorem}

It will be evident from the proof of Theorem \ref{Theor_Shadowing but not Generalized Hyperbolic} that its conclusion holds more generally. Namely, it is true for every Banach space $\mc{B}$ satisfying $\mc{B}\cong\mc{B}\oplus\mc{B}$ and admitting a closed non-complemented subspace such that $ \mc{B} / M \cong \mc{B}$.

\section{Proofs of Theorems \ref{Theor_Pseudo-Hyperbolicity Lipschitz Stability Shadowing} and \ref{Theor_Equivalence Shadowing and Spectrum}}\label{Sec_Stability}

Let $\mc B$ be a Banach space. Assertions~{\rm(1)} and~{\rm(2)} of Theorem
\ref{Theor_Pseudo-Hyperbolicity Lipschitz Stability Shadowing}
are proved separately.

\begin{remark}
In \cite[Theorem~1]{Bernardes_Messaoudi-A_Generalized_Grobman-Hartman Theorem}, it was shown that generalized hyperbolicity implies strong Lipschitz structural stability. In that proof, the conjugacy required by strong Lipschitz structural stability is obtained through an application of the Banach Fixed Point Theorem, made possible by the decomposition $\mc{B}=M\oplus N$ arising from generalized hyperbolicity. The method used to prove Assertion {\rm(1)} of Theorem \ref{Theor_Pseudo-Hyperbolicity Lipschitz Stability Shadowing} is different, since such an invariant decomposition is not necessarily available in the pseudo-hyperbolic setting.
\end{remark}

\smallskip

\begin{proof}[Proof of Theorem \ref{Theor_Pseudo-Hyperbolicity Lipschitz Stability Shadowing} {\rm(1)}]

Let $T\in\mc{GL}\lc\mc{B}\rc$ be a pseudo-hyperbolic operator, and let $P:\mc{B}\to\mc{B}$, $C>0$, and $\beta\in\lc0,1\rc$ satisfy \eqref{Eq_Pseudo-Hyperbolicic Exponential Bounds}. Set $Q=I-P$.

\paragraph{Proof of strong Lipschitz structural stability.}

Consider the complex Banach space
$$
U_b\lc\mc{B}\rc
=
\lfp
f:\mc{B}\to\mc{B}
:\;
f \text{ is bounded and uniformly continuous}
\rfp,
$$
equipped with the supremum norm.
The proof is divided into the following five lemmas.

\begin{lemma}\label{Lem_Right Inverse Psi}
Let $g:\mc{B}\to\mc{B}$ be a uniform homeomorphism. Then the linear operator 
$ \Psi_g:U_b\lc\mc{B}\rc\longrightarrow U_b\lc\mc{B}\rc$, 
defined by
\begin{equation}\label{Eq_Lemma RI Psi Map}
\Psi_g(f)=f\circ g-T\circ f,
\end{equation}
is bounded and admits a bounded linear right inverse $
W_g:U_b\lc\mc{B}\rc\longrightarrow U_b\lc\mc{B}\rc $ 
given by
\begin{equation}\label{Eq_W Right Inverse}
W_g(f)(x)
=
\sum_{k=0}^{\infty}
T^kP\bigl(f(g^{-k-1}(x))\bigr)
-
\sum_{k=1}^{\infty}
T^{-k}Q\bigl(f(g^{k-1}(x))\bigr)\text{, } \qquad \forall x\in \mc{B}.
\end{equation}
Moreover,
\begin{equation}\label{Eq_W Bound}
\|W_g\|
\le
C\frac{1+\beta}{1-\beta}.
\end{equation}

\end{lemma}

\paragraph{}

Since $T$ is invertible, let $ c=\|T^{-1}\|^{-1} > 0$.  
Fix $0<\e<1$ and define
\begin{equation}\label{Eq_epsilon-delta}
\delta
:=
\min
\lfp
{c \over 3},
\frac{\e(1-\beta)}{C(1+\beta)}
\rfp.
\end{equation}
The next classical lemma (see \cite[Theorem 1.5]{Hirsch-Pugh-Stable_Manifolds}) shows that, for every
$\phi\in\mf{Lip}_{\delta}\lc\mc{B}\rc$, the map $T+\phi$ is a uniform homeomorphism.

\begin{lemma}\label{Lem_T+phi Uniform Homeomorphism}
Let $T\in\mc{GL}(\mc{B})$ and $c = \ldav T^{-1} \rdav^{-1} > 0$. If
$\phi\in\mf{Lip}_{\delta}\lc\mc{B}\rc$
with $0 < \delta < c / 2 $, then the map
$T+\phi:\mc{B}\to\mc{B}$
is a bi-Lipschitz homeomorphism. In particular,
$T+\phi$
is a uniform homeomorphism.
\end{lemma}

\paragraph{} 

Fix $\phi\in\mf{Lip}_{\delta}\lc\mc{B}\rc$ and define $ V_{\phi}:U_b\lc\mc{B}\rc
\longrightarrow
U_b\lc\mc{B}\rc $ 
by
$$
V_{\phi}(f)
=
W_T\bigl(\phi\circ(I+f)\bigr),
$$
where $W_T:U_b\lc\mc{B}\rc\longrightarrow U_b\lc\mc{B}\rc$ denotes the right inverse of $\Psi_T$ provided by Lemma \ref{Lem_Right Inverse Psi}. The map $V_{\phi}$ is a contraction with Lipschitz constant $ \textrm{Lip}(V_{\phi})
\le
\e $. 
Indeed, for every $f,g\in U_b\lc\mc{B}\rc$,
\begin{align*}
\ldav
W_T\bigl(\phi\circ(I+f)\bigr)
-
W_T\bigl(\phi\circ(I+g)\bigr)
\rdav_{\infty} \underset{\eqref{Eq_W Bound}}{\le}
C
\frac{1+\beta}{1-\beta}
\,
\delta
\,
\ldav
f-g
\rdav_{\infty} \underset{\eqref{Eq_epsilon-delta}}{\le}
\e
\,
\ldav
f-g
\rdav_{\infty}.
\end{align*}
Hence, the Banach Fixed Point Theorem yields a unique
$h\in U_b\lc\mc{B}\rc$
satisfying
\begin{equation}\label{eq_h}
W_T\bigl(\phi\circ(I+h)\bigr)
=
h.
\end{equation}

\begin{lemma}\label{Lem_Fix Point of V}
Let $h\in U_b\lc\mc{B}\rc$ denote the unique fixed point of $V_{\phi}$. Then
\begin{equation}\label{Eq_Conjugation Lipschitz Stability}
\|h\|_{\infty}
\le
\e
\qquad\text{and}\qquad
(I+h)\circ T
=
(T+\phi)\circ(I+h).
\end{equation}
\end{lemma}

\paragraph{}

In view of relation \eqref{Eq_Conjugation Lipschitz Stability}, it remains to show that the map $H=I+h:\mc{B}\to\mc{B}$ is a homeomorphism. By Lemma \ref{Lem_T+phi Uniform Homeomorphism}, the map $T+\phi$ is a uniform homeomorphism. Hence, Lemma \ref{Lem_Right Inverse Psi} applies with $g=T+\phi$. Set
$$
h'=W_{T+\phi}(-\phi)\in U_b\lc\mc{B}\rc
\quad\text{and}\quad
H'=I+h'.
$$
Then
$$
\Psi_{T+\phi}(h')=-\phi
\quad\text{and}\quad
\ldav h'\rdav_{\oo}\le\e.
$$
Indeed, the first identity follows from Lemma \ref{Lem_Right Inverse Psi}, whereas the second inequality follows from
$$
\ldav h'\rdav_{\oo}
\le
\ldav W_{T+\phi}\rdav
\cdot
\ldav\phi\rdav_{\oo}
\underset{\eqref{Eq_W Bound}}{\le}
C\frac{1+\beta}{1-\beta}
\delta
\underset{\eqref{Eq_epsilon-delta}}{\le}
\e.
$$

\paragraph{}

The next lemma shows that $H'$ is the inverse of $H$.

\begin{lemma}\label{Lem_Conjugacy Inverse}
The maps $H=I+h$ and $H'=I+h'$ are inverses of each other. Equivalently,
$$
H\circ H'
=
H'\circ H
=
I.
$$
\end{lemma}

\paragraph{} 
By Lemmas \ref{Lem_Fix Point of V} and \ref{Lem_Conjugacy Inverse}, the operators $T$ and $T+\phi$ are topologically conjugate by the uniform homeomorphism $H$, which satisfies $ \ldav H-I\rdav_{\oo}\le\e$. 
Since the choice of $\e>0$ and of $\phi\in\mf{Lip}_{\delta}\lc\mc{B}\rc$, where $\delta$ is given by \eqref{Eq_epsilon-delta}, was arbitrary, it follows that $T$ is strongly Lipschitz structurally stable.

\paragraph{Proof of topological stability.}

Fix $\e>0$ and define
\begin{equation}\label{Eq_Theor PH-LS-S Topological Stability delta}
\delta
=
\frac{1-\beta}{(1+\beta)C}\,\e.
\end{equation}
Let $g:\mc{B}\to\mc{B}$ be a homeomorphism satisfying $ \ldav T-g\rdav_{\oo}\le\delta$. 

\smallskip 

Denote by
$$
C_b\lc\mc{B}\rc
=
\lfp
f:\mc{B}\to\mc{B}
:
f \text{ is continuous and bounded}
\rfp
$$
the Banach space of continuous and bounded maps on $\mc{B}$.
The proof of Lemma \ref{Lem_Right Inverse Psi} carries over verbatim to the space
$C_b\lc\mc{B}\rc$, yielding bounded operators
$$
\Psi_g,W_g:C_b\lc\mc{B}\rc\longrightarrow C_b\lc\mc{B}\rc,
$$
defined by \eqref{Eq_Lemma RI Psi Map} and \eqref{Eq_W Right Inverse}, respectively, such that
$$
\Psi_g\circ W_g=I
\quad\text{and}\quad
\ldav W_g\rdav
\le
C\frac{1+\beta}{1-\beta}.
$$

Set
$$
H
:=
I+W_g(T-g).
$$
The map $H$ is well defined since $T-g\in C_b\lc\mc{B}\rc$. Moreover,
$$
\ldav H-I\rdav_{\oo} 
\le
\ldav W_g\rdav
\,
\ldav T-g\rdav_{\oo}
\underset{ \eqref{Eq_W Bound}, \eqref{Eq_Theor PH-LS-S Topological Stability delta}}{ \le } 
\e, 
$$
and 
\begin{equation*}
T-g 
=
\Psi_g(H-I) 
\underset{\eqref{Eq_Lemma RI Psi Map}}{=}
H\circ g-g-T\circ H+T,
\end{equation*}
which implies that $ H\circ g
=
T\circ H$. 
Hence, $T$ is topologically stable.

This completes the proof of Point~\emph{(1)} of Theorem \ref{Theor_Pseudo-Hyperbolicity Lipschitz Stability Shadowing}. 
\end{proof}

\paragraph{}

For the rest of this section, given an operator $T \in \mc{GL}\lc \mc{B} \rc$, denote by $\mc{B}^{*}$ the dual space of $\mc{B}$ and by $T^{*}:\mc{B}^{*}\to\mc{B}^{*}$ the adjoint operator of $T$. It is well known (see, for instance, \cite[Theorem 1.10 and p.~7]{Aiena_Fredholm_Local_Spectral_Theory}) that
\begin{equation}\label{Eq_Dual Spectra}
    \sigma_{sur}\lc T \rc
    =
    \sigma_{ap}\lc T^{*} \rc .
\end{equation}

\begin{proof}[Proof of Point (2) of Theorem \ref{Theor_Pseudo-Hyperbolicity Lipschitz Stability Shadowing}]

Assume that $T$ is strongly Lipschitz structurally stable and suppose, for a contradiction, that $T$ does not have the shadowing property. By Theorem \ref{Theor_Equivalence Shadowing and Spectrum}, there exists $\lambda\in\SP^{1}$ such that $ \lambda\in\sigma_{\mathrm{sur}}(T)
=
\sigma_{\mathrm{ap}}\lc T^{*} \rc $. 

Set $\e=1$, and let $\delta>0$ be the constant provided by the definition of strong Lipschitz structural stability. Choose $\alpha>0$ and $N\in\N$ such that
\begin{equation}\label{Eq_Theor PH-LS-S alpha and N}
2\alpha<\delta
\qquad\text{and}\qquad
\alpha N>8.
\end{equation}
Define
\begin{equation}\label{Eq_Theor PH-LS-S K and L}
L
:=
\sum_{k=0}^{N-1}\ldav T^k\rdav
\qquad\text{and}\qquad
K
:=
\max_{0\le n\le N}\ldav T^n\rdav,
\end{equation}
and choose $\eta>0$ so that
\begin{equation}\label{Eq_Theor PH-LS-S eta definition}
8\eta L<\frac12
\qquad\text{and}\qquad
\eta(8K+1)<\frac{\alpha}{4}.
\end{equation}
Since $\lambda \in \sigma_{ap}\lc T^{*} \rc$, there exists a functional $f\in\mc{B}^*$ satisfying
\begin{equation}\label{Eq_Theor PH-LS-S Functional f}
\|f\|=1
\qquad\text{and}\qquad
\|T^*f-\lambda f\|<\eta.
\end{equation}

\paragraph{}

The remainder of the proof is devoted to deriving a contradiction from \eqref{Eq_Theor PH-LS-S Functional f}. Choose $x_0\in\mc{B}$ such that
\begin{equation}\label{Eq_Theor PH-LS-S x0}
f(x_0)=1
\qquad\text{and}\qquad
\ldav x_0\rdav_{\mc{B}}<2,
\end{equation}
and define the map $\phi:\mc{B}\to\mc{B}$ by
\begin{equation}\label{Eq_Theor PS-LS-S Lipschitz Map}
\phi(x)
:=
\alpha\lambda A\lc f(x) \rc \,x_0,
\qquad
x\in\mc{B},
\end{equation}
where $A:\mathbb C\to\mathbb C$ is the bounded Lipschitz map
\[
A(z)
=
\begin{cases}
\dfrac{z}{2},
&
|z|\le2,
\\[6pt]
\dfrac{z}{|z|},
&
|z|>2.
\end{cases}
\]

To verify that $\phi\in\mf{Lip}_{\delta}(\mc{B})$, observe that a standard argument shows that
$|A(z)|\le1$ for all $z\in\mathbb{C}$ and that
$\operatorname{Lip}(A)\le1$. Consequently,
\[
\ldav\phi\rdav_{\infty} 
\le 
2\alpha
<
\delta,
\]
and
\[
\operatorname{Lip}(\phi)
\le
2\alpha
<
\delta.
\]
Hence,
$\phi\in\mf{Lip}_{\delta}\lc\mc{B}\rc$.

Since $T$ is strongly Lipschitz structurally stable, there exists a homeomorphism
$H:\mc{B}\to\mc{B}$ satisfying
\begin{equation}\label{Eq_Theor PS-LS-S Conjugation Point 2}
H\circ T
=
(T+\phi)\circ H
\qquad\text{and}\qquad
\ldav H-I\rdav_{\infty}\le1.
\end{equation}

\smallskip

The contradiction follows upon establishing that
\begin{equation}\label{Eq_Theor PH-LS-S Contradiction}
\lav f\lc H(T^{N}z_{0})\rc\rav>6
\quad\text{and}\quad
\frac52
<
\lav f\lc H(T^{n}z_{0})\rc\rav
<
\frac{11}{2},
\qquad
0\le n \le  N,
\end{equation}
where $z_{0}=4x_{0}$.

To this end, notice that, for every $0\le n\le N$,
\begin{align*}
\bigl|f(H(T^{n}z_{0}))-4\lambda^{n}\bigr| 
\underset{\eqref{Eq_Theor PH-LS-S K and L},
\eqref{Eq_Theor PH-LS-S Functional f}}{\le}
1+8\eta L
\underset{\eqref{Eq_Theor PH-LS-S eta definition}}{<}
\frac32.
\end{align*}
Since $|\lambda|=1$, the right-hand side inequality in
\eqref{Eq_Theor PH-LS-S Contradiction}
follows immediately from the triangle inequality.

The left-hand side inequality in
\eqref{Eq_Theor PH-LS-S Contradiction}
is a consequence of
\begin{equation}\label{Eq_Theor PH-LS-S Iteration}
\lav f(H(Tx))\rav
\underset{ \eqref{Eq_Theor PH-LS-S Functional f}, \eqref{Eq_Theor PS-LS-S Lipschitz Map} , \eqref{Eq_Theor PS-LS-S Conjugation Point 2} }{\ge} 
\lav f(H(x))\rav
+
\alpha
-
\eta
\bigl(
1+\ldav x\rdav_{\mc B}
\bigr),
\end{equation}
whenever
$
\lav f(H(x))\rav\ge2.
$
Indeed, the right-hand side inequality in
\eqref{Eq_Theor PH-LS-S Contradiction}
implies that
\[
\lav f(H(T^{n}z_{0}))\rav>2,
\qquad
0\le n\le N-1,
\]
while
\[
\ldav T^{n}z_{0}\rdav_{\mc B}
\underset{\eqref{Eq_Theor PH-LS-S K and L}}{\le}
8K,
\qquad
0\le n\le N-1.
\]
Applying
\eqref{Eq_Theor PH-LS-S Iteration}
with
$x=T^{n}z_{0}$,
one obtains
\[
\lav f(H(T^{n+1}z_{0}))\rav
\underset{\eqref{Eq_Theor PH-LS-S eta definition}}{\ge}
\lav f(H(T^{n}z_{0}))\rav
+\frac{3\alpha}{4},
\qquad
0\le n\le N-1.
\]
Iterating the previous inequality yields
\[
\lav f(H(T^{N}z_{0}))\rav
\ge
\lav f(H(z_{0}))\rav
+\frac{3\alpha N}{4}
\underset{\eqref{Eq_Theor PH-LS-S alpha and N}}{>}
6,
\]
thus establishing the left-hand side inequality in
\eqref{Eq_Theor PH-LS-S Contradiction}.

The inequalities in
\eqref{Eq_Theor PH-LS-S Contradiction}
contradict
\eqref{Eq_Theor PH-LS-S Functional f}.
Therefore,
$\SP^{1}\cap\sigma_{\mathrm{ap}}(T^{*})=\emptyset$,
and the proof is complete. 
\end{proof}

\paragraph{} 

For the remainder of this section, the positive shadowing property will simply be referred to as the shadowing property.

Theorem \ref{Theor_Equivalence Shadowing and Spectrum} is an immediate consequence of the following proposition. To state it, let $\mathbb{L}$ denote either $\Z$ or $\N$, and consider the Banach space
$$ 
\ell^{\oo}_{\mathbb{L}}(\mc{B})
=
\left\{
(x_n)_{n\in\mathbb{L}}
\in
\mc{B}^{\mathbb{L}}
:
\sup_{n\in\mathbb{L}}
\|x_n\|_{\mc{B}}
<
+\infty
\right\}.
$$ 
Thus, if $\mathbb{L}=\Z$, then $\ell^{\oo}_{\mathbb{L}}(\mc{B})$ is the space of bounded bi-infinite sequences in $\mc{B}$, whereas if $\mathbb{L}=\N$, it is the space of bounded one-sided sequences.

\begin{proposition}\label{Prop_Shadowing, Spectra, and Surjectivity}
Let $\mc B$ be a Banach space. Assume that either
\[
\begin{aligned}
&T\in\mc{GL}\lc\mc{B}\rc
\text{ and }
\mathbb{L}=\Z,
&&\text{or}&&
T\in\mc{L}\lc\mc{B}\rc
\text{ and }
\mathbb{L}=\N.
\end{aligned}
\]
Define the operator
\begin{equation*}\label{Eq_Def H operator}
\begin{aligned}
H_{T}:\,
&\ell^{\oo}_{\mathbb{L}}\lc \mc{B} \rc
\longrightarrow
\ell^{\oo}_{\mathbb{L}}\lc \mc{B} \rc,
\\
&H_{T}\lc (x_{n})_{n\in \mathbb{L}} \rc
=
\lc (x_{n+1}-Tx_{n}) \rc_{n\in \mathbb{L}}.
\end{aligned}
\end{equation*}
Then the following assertions are equivalent:
\begin{enumerate}
\item The operator $T$ has the shadowing property.

\item The operator $H_{T}$ is surjective.

\item The operator $T-\lambda I$ is surjective for every $\lambda\in\SP^{1}$.
\end{enumerate}
\end{proposition}

\begin{remark}\label{Rem_Shadowing and Surjectivity} 
The equivalence between {\rm(1)} and {\rm(2)} in Proposition
\ref{Prop_Shadowing, Spectra, and Surjectivity}
has appeared previously in the literature; see
\cite[Lemma~10]{Bernardes-Expansivity_and_Shadowing_in_Linear_Dynamics},
\cite[Proposition~3]{Cirilo_Gollobit_Pujals-Dynamics_of_Generalized_Hyperbolic_Linear_Operators},
and \cite{Pilyugin-Shadowing_in_Dynamical_Systems}. 
The equivalence between {\rm(1)} and {\rm(3)} was recently established by
Dragičević and Pituk
\cite{Dragicevic_Pituk-Duality_between_Shadowing_and_Uniform_Expansivity_in_Linear_Dynamics}.
For completeness, an alternative and more elementary proof of the latter equivalence will be obtained by proving the equivalence between {\rm(2)} and {\rm(3)}.
\end{remark}

The proof of Proposition \ref{Prop_Shadowing, Spectra, and Surjectivity} relies on the following result due to Davis \cite{Davis-Rosenthal_Solving_Linear_Operator_Equations}.

\begin{lemma}\label{Lem_Spectrum of the Difference}
\cite[Theorem 2]{Davis-Rosenthal_Solving_Linear_Operator_Equations}
Let $\mc{B}$ be a complex Banach space, and let $T_{1},T_{2}:\mc{B}\to\mc{B}$ be commuting bounded linear operators. Then
\begin{equation*}
    \sigma_{ap}\lc T_{1}-T_{2}\rc
    \sub 
    \left\{
    \alpha-\beta
    :
    \alpha\in\sigma_{ap}\lc T_{1}\rc,
    \;
    \beta\in\sigma_{ap}\lc T_{2}\rc
    \right\}.
\end{equation*}
In particular, if $ 0\in\sigma_{ap}\lc T_{1}-T_{2}\rc$, then $ \sigma_{ap}\lc T_{1}\rc
\cap
\sigma_{ap}\lc T_{2}\rc
\neq
\emptyset$. 
\end{lemma}

\bigskip

\begin{proof}[Proof of Proposition \ref{Prop_Shadowing, Spectra, and Surjectivity}]

Choose $\mathbb{L}\in\{\Z,\N\}$, and let
$T:\mc{B}\to\mc{B}$ and
$H_T:\ell^\infty_{\mathbb{L}}(\mc{B})\to\ell^\infty_{\mathbb{L}}(\mc{B})$
be as in the statement of the proposition.
In view of Remark \ref{Rem_Shadowing and Surjectivity},
it remains to establish the equivalence between {\rm(2)} and {\rm(3)}. 

\paragraph{} 

It is first shown that surjectivity of $H_T$ implies surjectivity of
$T-\lambda I$ for every $\lambda\in\SP^{1}$.
To this end, it suffices to establish the following claim:
\begin{equation}\label{Eq_Claim H Surjective implies T-I Surjective}
\text{If } H_{T} \text{ is surjective, then } T-I \text{ is surjective.}
\end{equation}
Indeed, fix $\lambda\in\SP^{1}$. Since $T$ has the shadowing property if and only if $\lambda^{-1}T$ has the shadowing property, the equivalence of {\rm(1)} and {\rm(2)} in Proposition
\ref{Prop_Shadowing, Spectra, and Surjectivity}
implies that $H_T$ is surjective if and only if $H_{\lambda^{-1}T}$ is surjective.
Applying Claim
\eqref{Eq_Claim H Surjective implies T-I Surjective}
to $\lambda^{-1}T$, one concludes that
$\lambda^{-1}T-I$ is surjective, or equivalently, that
$T-\lambda I$ is surjective.

\paragraph{}

To establish Claim
\eqref{Eq_Claim H Surjective implies T-I Surjective},
denote by
$B\lc\bm0,r\rc$ and $B_{\oo}\lc\bm0,r\rc$
the open balls of radius $r>0$ centered at $\bm0$ in the spaces
$\mc{B}$ and $\ell^{\oo}_{\mathbb{L}}\lc\mc{B}\rc$, respectively.
Since $H_{T}$ is surjective, the Open Mapping Theorem
\cite[Theorem 2.6]{Brezis_Functional_Analysis} 
guarantees the existence of $\rho>0$ such that
\begin{equation}\label{Eq_H Operator Open}
B_{\oo}\lc\bm0,\rho\rc
\sub
H_{T}\lc B_{\oo}\lc\bm0,1\rc\rc.
\end{equation} 
By \cite[p. 36, Proof of Theorem 2.6, Step 2]{Brezis_Functional_Analysis}, it suffices to establish 
\begin{equation}\label{Eq_I-T Surjective}
B\lc\bm0,\rho\rc
\sub
\overline{\lc I-T\rc\lc B\lc\bm0,1\rc\rc}.
\end{equation}
Fix $y\in\mc{B}$ satisfying
$\ldav y\rdav_{\mc B}<\rho$, and let
$\bm y=(y)_{n\in\mathbb L}\in\ell^{\oo}_{\mathbb L}\lc\mc B\rc$
denote the constant sequence with value $y$.
By \eqref{Eq_H Operator Open}, there exists
$\lc x_n\rc_{n\in\mathbb L}\in\ell^{\oo}_{\mathbb L}\lc\mc B\rc$
such that
\[
x_{n+1}-Tx_n=y
\quad\text{and}\quad
\ldav x_n\rdav_{\mc B}<1,
\qquad n\in\mathbb L.
\]
Define
\[
z_k
=
\frac1k
\sum_{m=1}^{k}
x_m,
\qquad
k\ge1.
\]
Then
$\ldav z_k\rdav_{\mc B}<1$ and $ (I-T)z_k =
\lc x_{1} - x_{k+1} \rc/k  + y $, for all $k \ge 1$. 
Letting
$k\to+\infty$
yields $ y \in \overline{(I-T)(B(\bm0,1))}$. 
This proves Claim
\eqref{Eq_Claim H Surjective implies T-I Surjective}.

\paragraph{}

It remains to prove the converse implication.
Assume that $ \sigma_{sur}\lc T \rc \cap \SP^{1}=\emptyset$, 
and suppose, for a contradiction, that the operator $ H_{T}:\ell^{\oo}_{\mathbb{L}}\lc\mc{B}\rc
\longrightarrow
\ell^{\oo}_{\mathbb{L}}\lc\mc{B}\rc$ 
is not surjective. Then $0\in\sigma_{sur}\lc H_{T}\rc$.

Notice that
\begin{equation}\label{Eq_Lem Sh Spectrum Operator Decomposition}
H_T
=
S-\widehat T,
\end{equation}
where the operators
\[
S,\widehat T:
\ell^\infty_{\mathbb L}(\mc B)
\longrightarrow
\ell^\infty_{\mathbb L}(\mc B)
\]
are defined by
\[
S\bigl((x_n)_{n\in\mathbb L}\bigr)
=
(x_{n+1})_{n\in\mathbb L},
\qquad
\widehat T\bigl((x_n)_{n\in\mathbb L}\bigr)
=
(Tx_n)_{n\in\mathbb L}.
\] 
Moreover,
\begin{equation}\label{Eq_Lem Sh Spectrum T}
\sigma_{sur}(\widehat T)
=
\sigma_{sur}(T).
\end{equation}
Indeed, the inclusion $\sigma_{sur}(T)\subseteq\sigma_{sur}(\widehat T) $ is immediate. Conversely, let $ \lambda\notin\sigma_{sur}(T)$. 
By the Open Mapping Theorem
\cite[Theorem~2.6]{Brezis_Functional_Analysis},
there exists a constant 
\(r_\lambda>0\)
such that $ B(\mathbf0,r_\lambda)
\subseteq
(T-\lambda I)\bigl(B(\mathbf0,1)\bigr)$. 
Applying this relation coordinatewise, it follows that for every
\(
(y_n)_{n\in\mathbb L}
\in
\ell^\infty_{\mathbb L}(\mc B)
\)
there exists
\(
(x_n)_{n\in\mathbb L}
\in
\ell^\infty_{\mathbb L}(\mc B)
\)
such that
\[
(\widehat T-\lambda\widehat I)(x_n)_{n\in\mathbb L}
=
(y_n)_{n\in\mathbb L}  \quad \text{and} \quad \|(x_n)_{n\in\mathbb L}\|_\infty
\le
\frac{1}{r_\lambda}
\|(y_n)_{n\in\mathbb L}\|_\infty ,
\]
where \(\widehat I\) denotes the identity operator on
\(\ell^\infty_{\mathbb L}(\mc B)\). 
Hence,
\(
\widehat T-\lambda\widehat I
\)
is surjective, proving that $ \lambda\notin\sigma_{sur}(\widehat T)$.

Since $ S\widehat{T}
=
\widehat{T}S$, 
their adjoints also commute. Moreover, by
\eqref{Eq_Dual Spectra},
\[
0
\in
\sigma_{sur}\lc H_{T}\rc
=
\sigma_{ap}\lc H_{T}^{*}\rc
\underset{ \eqref{Eq_Lem Sh Spectrum Operator Decomposition} }{=} 
\sigma_{ap}\lc S^{*}-\widehat{T}^{*}\rc.
\]
Hence, Lemma~\ref{Lem_Spectrum of the Difference} yields $ 
\sigma_{ap}\lc S^{*}\rc
\cap
\sigma_{ap}\lc\widehat{T}^{*}\rc
\neq
\emptyset$. 
Since
$$ 
\sigma_{ap}\lc\widehat{T}^{*}\rc
\underset{  \eqref{Eq_Dual Spectra} , \eqref{Eq_Lem Sh Spectrum T} }{=} 
\sigma_{ap}\lc T^{*}\rc
\underset{ \eqref{Eq_Dual Spectra} }{=}
\sigma_{sur}\lc T\rc \quad \text{and} \quad 
\sigma_{ap}\lc S^{*}\rc
\underset{ \eqref{Eq_Dual Spectra} }{=}
\sigma_{sur}\lc S\rc,
$$
one obtains $ \sigma_{sur}\lc S\rc
\cap
\sigma_{sur}\lc T\rc
\neq
\emptyset$. 
Finally, it is readily verified that
$\sigma_{sur}(S)=\SP^{1}$.
Hence,
$\sigma_{sur}(T)\cap\SP^{1}\neq\emptyset$,
contradicting the hypothesis and completing the proof. 
\end{proof}

\section{Proof of Proposition \ref{Prop_Pseudo-Hyperbolicity Spectral Characterization} and Theorem \ref{Theor_Banach Equivalence Pseudo-Hyperbolicity Lipschitz Stability Shadowing}}\label{Sec_Pseudo-Hyperbolicity and Spectrum}

In this section, Proposition 
\ref{Prop_Pseudo-Hyperbolicity Spectral Characterization}
and Theorem 
\ref{Theor_Banach Equivalence Pseudo-Hyperbolicity Lipschitz Stability Shadowing}
are proved, in this order.

\begin{proof}[Proof of Proposition \ref{Prop_Pseudo-Hyperbolicity Spectral Characterization}]

Let $\mc{B}$ be a Banach space and let
$T:\mc{B}\to\mc{B}$
be an invertible bounded linear operator.
The implication from the spectral assumption $ \sigma_r(T)\cap\SP^1=\emptyset $ to the existence of a bounded linear operator
$P:\mc{B}\to\mc{B}$
satisfying the exponential bounds
\eqref{Eq_Pseudo-Hyperbolicic Exponential Bounds}
was established by Pituk
\cite[Proposition 2.5 and p.~10]{Pituk_Spectral Characterization of Shadowing for Linear Operators on Hilbert Spaces}.
Therefore, it remains only to prove the converse implication.

\medskip 

Assume that $T$ is pseudo-hyperbolic and fix
$\lambda\in\SP^{1}$.
To prove that
$\sigma_r(T)\cap\SP^{1}=\emptyset$,
it is enough to construct an operator
$R_{\lambda}:\mc{B}\to\mc{B}$
satisfying
$$ \lc \lambda I-T \rc R_{\lambda}=I. $$

Let
$P:\mc{B}\to\mc{B}$
be the bounded linear operator appearing in the definition of pseudo-hyperbolicity. Then there exist constants
$C>0$
and
$0<\beta<1$ satisfying the exponential bounds \eqref{Eq_Pseudo-Hyperbolicic Exponential Bounds}. Hence, 
the operator
\begin{equation}\label{Eq_Right Inverse of Generalized Hyperbolic}
R_{\lambda}(x)
=
\lambda^{-1}
\sum_{n=0}^{\infty}
\lambda^{-n}T^nP(x)
-
\lambda^{-1}
\sum_{n=1}^{\infty}
\lambda^nT^{-n}(I-P)(x)
\end{equation}
is well defined. Indeed, since $\lambda \in \SP^{1}$, both series converge absolutely by
\eqref{Eq_Pseudo-Hyperbolicic Exponential Bounds}.

A straightforward computation shows that
$(\lambda I-T)R_{\lambda}=I$,
thereby completing the proof. 
\end{proof}

\paragraph{}

For the proof of Theorem \ref{Theor_Banach Equivalence Pseudo-Hyperbolicity Lipschitz Stability Shadowing}, the following standard characterization of surjective operators will be used (see, for instance, \cite[Theorem A.10]{Aiena_Fredholm_Local_Spectral_Theory} and \cite[Theorem 2.12]{Brezis_Functional_Analysis}). If $\mc{B}$ is a Banach space and $T:\mc{B}\to\mc{B}$ is a bounded surjective linear operator, then
\begin{equation}\label{Eq_Right Inverse and Closed Complement}
T \text{ admits a bounded linear right inverse if and only if } \ker T \text{ admits a closed complement.}
\end{equation}

\begin{proof}[Proof of Theorem \ref{Theor_Banach Equivalence Pseudo-Hyperbolicity Lipschitz Stability Shadowing}]

Let $T \in \mc{GL}\lc \mc{B} \rc$ be an invertible operator on the Banach space $\mc{B}$. Assume that $T$ satisfies condition \eqref{Eq_Assumption Closed Complement}. Then, in view of equivalence \eqref{Eq_Right Inverse and Closed Complement}, one infers
that
$$\sigma_{r}\lc T \rc \cap \SP^{1} = \sigma_{sur}\lc T \rc \cap \SP^{1} .$$ 
The equivalence of Points~(1)--(4) now follows directly from
\cite[Corollary~15]{Bernardes_Caraballo_Darji_Favaro_Peris-Generalized_Hyperbolicity_Stability_and_Expansivity_for_Operators_on_Locally_Convex_Spaces},
Theorem~\ref{Theor_Equivalence Shadowing and Spectrum}, Theorem~\ref{Theor_Pseudo-Hyperbolicity Lipschitz Stability Shadowing}, and Proposition~\ref{Prop_Pseudo-Hyperbolicity Spectral Characterization}. 

\medskip

Finally, in the special case where $\mc{B}$ is a Hilbert space, condition \eqref{Eq_Assumption Closed Complement} is automatically satisfied, since every closed subspace of a Hilbert space admits an orthogonal complement.

The proof is complete. 
\end{proof}

\bigskip 

\begin{remark}
Let $\mathcal B$ be a complex Banach space. Denote by $\textrm{GH}
\lc \mathcal B \rc,
\textrm{PH}\lc \mathcal B \rc,
\textrm{S}\lc \mathcal B \rc$, the sets of generalized hyperbolic operators, pseudo-hyperbolic operators, and operators with the shadowing property on $\mathcal B$, respectively. Then, with respect to the operator norm topology on $\mathcal L(\mathcal B)$, these sets are open and satisfy
$$
\textrm{GH}
\lc \mathcal B \rc 
\sub 
\textrm{PH}\lc \mathcal B \rc
\sub
\textrm{S}\lc \mathcal B \rc.
$$

\end{remark}

\section{Proof of Theorem \ref{Theor_Shadowing but not Generalized Hyperbolic}}\label{Sec_Example Shadowing but not Generalized Hyperbolic}

The goal of this section is to construct an invertible operator on $\ell^{p}\lc \N \rc$, $p \in \lc 1, +\oo \rc \backslash \lfp 2 \rfp $, with the shadowing property that is not generalized hyperbolic. In view of Theorem 
\ref{Theor_Equivalence Shadowing and Spectrum}
and Proposition 
\ref{Prop_Pseudo-Hyperbolicity Spectral Characterization},
it suffices to find an invertible operator $T$ on $ \ell^{p}\lc\N\rc$ such that \begin{equation}\label{Eq_Shadowing but not Generalized Hyperbolic}
\sigma_{sur}(T)\cap\SP=\emptyset
\qquad\text{and}\qquad
\sigma_r(T)\cap\SP\neq\emptyset.
\end{equation}
To this end, the following proposition provides a sufficient condition for the existence of an operator $T$ satisfying \eqref{Eq_Shadowing but not Generalized Hyperbolic}.

\begin{proposition}\label{Prop_Spaces with Shadowing but not Generalized Hyperbolicity}
Let $\mc{B}$ be a complex Banach space such that
\begin{equation}\label{Eq_B equal B+B}
\mc{B}\cong \mc{B}\oplus \mc{B}.
\end{equation}
Then there exists $T\in\mc{GL}(\mc B)$ satisfying \eqref{Eq_Shadowing but not Generalized Hyperbolic} if and only if there exists a closed non-complemented subspace $M \sub \mc B$ such that the quotient space 
\begin{equation}\label{Eq_Condition Isomorphic Quotient}
\mc B/M\cong\mc B.
\end{equation}
\end{proposition} 

\begin{remark}
In view of Proposition \ref{Prop_Spaces with Shadowing but not Generalized Hyperbolicity}, it is worth mentioning that, by the classical result of Lindenstrauss and Tzafriri \cite{Lindenstrauss_Tzafriri-On_the_Complemented_Subspaces_Problem}, a Banach space $\mc{B}$ admits a closed non-complemented subspace if and only if it is not isomorphic to a Hilbert space.
\end{remark}

\begin{proof}[Proof of Proposition \ref{Prop_Spaces with Shadowing but not Generalized Hyperbolicity}]
Assume that the Banach space $\mc{B}$ satisfies relation \eqref{Eq_B equal B+B}.

\medskip

Suppose first that there exists $T\in\mc{GL}(\mc B)$ satisfying
\eqref{Eq_Shadowing but not Generalized Hyperbolic}. Let
$\lambda\in\sigma_r(T)\cap\SP^1$ and set $ M:=\ker(T-\lambda I)$. 
By \eqref{Eq_Shadowing but not Generalized Hyperbolic} and equivalence
\eqref{Eq_Right Inverse and Closed Complement},
the subspace $M$ is non-complemented. Relation \eqref{Eq_Condition Isomorphic Quotient}
follows by applying the Isomorphism Theorem \cite[Theorem 1.7.14]{Megginson_Banach_Space_Theory} to the operator $T-\lambda I$, noting that $T-\lambda I$ is surjective and therefore has closed range.

\medskip 

Conversely, suppose that there exists a closed non-complemented subspace $M \sub \mc B$ satisfying relation \eqref{Eq_Condition Isomorphic Quotient}, and denote by $ \mc X=\mc B\oplus\mc B$. 
By relations \eqref{Eq_B equal B+B} and \eqref{Eq_Condition Isomorphic Quotient}, there exist isomorphisms
$U:\mc X\to\mc B$
and
$V:\mc B/M\to\mc B$.
Define the linear operator

\[
A:\mc B\to\mc B, \qquad A=V\circ q,
\]
where
$q:\mc B\to\mc B/M$
is the quotient map sending each $x\in\mc B$ to its equivalence class
$\lb x\rb_M$.
Then $A$ is surjective and
$\ker(A)=M$,
so $\ker(A)$ does not admit a closed complement in $\mc B$.

Define the invertible bounded linear operator
\[
\tilde{T}:\mc X\longrightarrow\mc X,
\qquad
\tilde{T}(u,v)
=
\bigl(v,u+\alpha Av\bigr),
\]
where $ \alpha= 4/\e $, and $\varepsilon>0$ is chosen so that
\[
A+cI
\text{ is surjective for every }
c\in\C
\text{ satisfying }
|c|<\varepsilon.
\]
Such a constant $ \e > 0$ exists because $A$ is surjective and the surjective resolvent set
$\rho_{sur}(A)=\C\setminus\sigma_{sur}(A)$
is open
\cite[Theorem~1.12]{Aiena_Fredholm_Local_Spectral_Theory}.

\paragraph{}
The goal is to prove that relation \eqref{Eq_Shadowing but not Generalized Hyperbolic} is satisfied for
$T=U\tilde{T}U^{-1}$.
Since $U$ is an isomorphism,
\[
\sigma_{sur}(T)=\sigma_{sur}(\tilde{T})
\qquad\text{and}\qquad
\sigma_{r}(T)=\sigma_{r}(\tilde{T}).
\]
Therefore, it suffices to verify relation
\eqref{Eq_Shadowing but not Generalized Hyperbolic}
for $\tilde{T}$.
In particular, it is enough to prove that
\[
1\in\sigma_{r}(\tilde{T})
\qquad\text{and}\qquad
\sigma_{sur}(\tilde{T})\cap\SP^{1}=\varnothing.
\]

Observe that
\[
\ker(\tilde{T}-I)
=
\left\{
(u,u):u\in\ker(A)
\right\}.
\]
Since $\ker(A)$ admits no closed complement in $\mc B$, the subspace
$\ker(\tilde{T}-I)$
admits no closed complement in $\mc X$.
Hence, by \eqref{Eq_Right Inverse and Closed Complement},
the operator $\tilde{T}-I$ admits no bounded linear right inverse.
Therefore,
$1\in\sigma_{r}(\tilde{T})$.

\bigskip 

It remains to prove that $ \sigma_{sur}(\tilde{T})\cap\SP^{1}=\varnothing$. 
Fix $\lambda\in\SP^{1}$.
To show that $\tilde{T}-\lambda I$ is surjective, let
$(r,s)\in\mc X$ be arbitrary.
Notice that
\[
\left|
\frac{\lambda^{-1}-\lambda}{\alpha}
\right|
\le
\frac{2}{\alpha}
<
\varepsilon.
\]
Hence, the operator $ A+  \alpha^{-1}\lc \lambda^{-1}-\lambda \rc I $ is surjective.
Since multiplication by the nonzero scalar $\alpha\lambda$
preserves surjectivity, the operator $\alpha\lambda A+ \lc 1-\lambda^{2} \rc I$ is also surjective.
Therefore, there exists $u\in\mc B$ such that
\[
\left(
\alpha\lambda A+ \lc 1-\lambda^{2} \rc I
\right)u
=
s-\alpha Ar+\lambda r.
\]
Define $ v=r+\lambda u$. A straightforward computation shows that
$ \lc \tilde{T}-\lambda I \rc \lc u,v  \rc = \lc r,s \rc$. 
Hence $\tilde{T}-\lambda I$ is surjective.
Since $\lambda\in\SP^{1}$ was arbitrary, $\sigma_{sur}\big( \tilde{T} \big) \cap\SP^{1}=\varnothing$. 

This proof is complete. 
\end{proof}

\bigskip

\begin{proof}[Proof of Theorem \ref{Theor_Shadowing but not Generalized Hyperbolic}]
Fix $p \in \lc 1,+\oo \rc$ with $p\neq2$, and notice that
$$ \ell^{p}\lc \N \rc\cong\ell^{p}\lc \N \rc\oplus\ell^{p}\lc \N \rc . $$ 
The proof is divided into two cases.

\smallskip 

Suppose first that $2<p<+\oo$. By \cite[p.~273]{Szankowski-Three_Space_Problems_Approximation_Property}, there exists a closed non-complemented subspace 
$M \sub \ell^{p}\lc\N\rc$
such that
$\ell^{p}\lc\N\rc/M\cong\ell^{p}\lc\N\rc$.
Therefore, by Proposition
\ref{Prop_Spaces with Shadowing but not Generalized Hyperbolicity},
there exists an operator
$T\in\mc{GL}\lc\ell^{p}\lc\N\rc\rc$
satisfying
\eqref{Eq_Shadowing but not Generalized Hyperbolic}.
Hence, by Theorem 
\ref{Theor_Equivalence Shadowing and Spectrum}
and Proposition 
\ref{Prop_Pseudo-Hyperbolicity Spectral Characterization},
$T$ has the shadowing property but is not generalized hyperbolic.

\medskip

Suppose now that $1<p<2$, and set $q=p/(p-1)>2$. By
\cite[Theorem~6]{Rosenthal-Subspaces_of_Lp_Independent_Random_Variables},
there exists a closed non-complemented subspace
$N \sub \ell^{q}\lc\mathbb N\rc$
that is isomorphic to
$\ell^{q}\lc\N\rc$.
Let
$F:\ell^{q}(\mathbb N)\to N$
be an isomorphism, let
$\iota_N:N\to\ell^{q}(\mathbb N)$
denote the inclusion map, and define $ J=\iota_NF$.

To conclude the proof, it suffices to show that
$M=\ker\lc J^{*}\rc$
is not complemented.
Since the kernel of a bounded linear operator is closed, Proposition
\ref{Prop_Spaces with Shadowing but not Generalized Hyperbolicity},
together with Theorem 
\ref{Theor_Equivalence Shadowing and Spectrum}
and Proposition 
\ref{Prop_Pseudo-Hyperbolicity Spectral Characterization},
yields the existence of an operator
$T\in\mc{GL}\lc\ell^{p}\lc\N\rc\rc$
that has the shadowing property but is not generalized hyperbolic,
provided that $M$ is not complemented.

To this end, by equivalence
\eqref{Eq_Right Inverse and Closed Complement},
it is enough to show that the operator
$J^{*}:\ell^{p}\lc\N\rc\to\ell^{p}\lc\N\rc$
admits no bounded right inverse.
Assume, for a contradiction, that there exists
$R\in\mc L\lc\ell^{p}\lc\N\rc\rc$
such that
$J^{*}R=I$.
Then, upon identifying $J^{**}$ with $J$, one has
$R^{*}J=I$.
Hence, the operator
$P=JR^{*}$
satisfies $ P^{2}=P $ and $ P\lc\ell^{q}\lc\N\rc\rc = N$. 
Thus $P$ is a projection onto $N$, contradicting the assumption that $N$ is non-complemented. This proves the claim and completes the proof.
 \end{proof}

\begin{remark}
Examples of bounded linear operators (not necessarily invertible) satisfying
\eqref{Eq_Shadowing but not Generalized Hyperbolic}
can be deduced from the work of Burlando
\cite{Burlando-Continuity_of_Spectrum_and_Spectral_Radius_in_Algebras_of_Operators}.
Indeed, 
\cite[Example 1.1]{Burlando-Continuity_of_Spectrum_and_Spectral_Radius_in_Algebras_of_Operators}
provides an operator $T$ satisfying
$$
0\notin\sigma_{sur}(T)
\qquad\text{and}\qquad
0\in\sigma_{r}(T).
$$ 
Since the surjective resolvent set $ \rho_{sur}(T)=\C\setminus\sigma_{sur}(T) $ is open \cite[Theorem 1.12]{Aiena_Fredholm_Local_Spectral_Theory}, there exists $c>0$ such that
$ B(0,c) \sub \rho_{sur}(T)$. It is then readily verified that the operator $ \widetilde T = \frac{4}{c}T+I $ satisfies \eqref{Eq_Shadowing but not Generalized Hyperbolic}.
\end{remark}

\paragraph{Acknowledgments} 
\thanks{Ali Messaoudi was partially supported by CNPq grant 310784/2021-2 and by FAPESP grant 2019/10269-3. 
José Tofanin Neto was supported by CAPES grant 88887.888179/2023-00. 
Manuel Saavedra was supported by  FAPESP grant 2026/08439-1.
Ioannis Tsokanos was supported by FAPESP grant 2024/10135-5. }

\end{document}